\documentclass[review]{elsarticle}

\usepackage{lineno,hyperref}
\modulolinenumbers[5]
\usepackage{amsmath, amsfonts,amsthm}
\newcommand\cplus{\mathbin{\raisebox{-\height}{$+$}}}
\newtheorem{theorem}{Theorem}

\newtheorem{corollary}{Corollary}
\journal{arXiv.org}

\begin{document}

\begin{frontmatter}

\title{Generalized continued fraction expansions for $\pi$ and $e$}

\author[mymainaddress]{Shirali Kadyrov}
\address[mymainaddress]{Department of Mathematics and Natural Sciences, Suleyman Demirel University, Kaskelen, Kazakhstan}
\ead{shirali.kadyrov@sdu.edu.kz}
\author[mymainaddress,mysecondaryaddress]{Farukh Mashurov}
\address[mysecondaryaddress]{Kazakh-British Technical University, Almaty, Kazakhstan}

%\address[mymainaddress]{Department of Mathematics and Natural Sciences, Suleyman Demirel University, Kaskelen, Kazakhstan}
\ead{farukh.mashurov@sdu.edu.kz}
%\author[mymainaddress]{Shirali Kadyrov}

%\ead{shirali.kadyrov@sdu.edu.kz}


\begin{abstract}
Recently Raayoni et al. announced various conjectures on  continued fractions of fundamental constants automatically generated with machine learning techniques. In this paper we prove some of their stated conjectures for Euler number $e$ and show the equivalence of some of the listed conjectures. Moreover, we propose a simple method that can be used to generate other continued fractions using their series representations.
\end{abstract}


\end{frontmatter}

\linenumbers

\section{Introduction}

Continued fraction theory is one of the old area of number theory that dates back to Bombelli from 16th Century \cite{jones1980continued}. The infinite continued fractions were first studied by Lord Brouncker who stated the well-known continued fraction for $4/\pi$. All giants of number theory including Euler, Gauss, Lagrange, Legendre, and Galois contributed to the field. For any given sequences $(a_n)$ and $(b_n)$ continued fraction is written as
\begin{equation}\label{eqn:cf}
b_0+\cfrac{a_1}{b_1+\cfrac{a_2}{b_2+\cfrac{a_3}{ b_3+ \ddots }}}\,, \text{  or  }b_0+\frac{a_1}{b_1}\cplus \frac{a_2}{b_2} \cplus \frac{a_3}{b_3} \cplus\cdots. 
\end{equation}
For any $n$, a natural number, \emph{$n$'th convergent} of \eqref{eqn:cf} is given by
\begin{equation*}
\frac{A_n}{B_n}:=b_0+\frac{a_1}{b_1}\cplus \frac{a_2}{b_2} \cplus \frac{a_3}{b_3} \cplus\cdots \cplus \frac{a_n}{b_n}.
\end{equation*}
We say that a given number $x$ has \emph{continued fraction expansion} \eqref{eqn:cf} provided 
$$\lim_{n \to \infty} \frac{A_n}{B_n}=x.$$
Inductively it is easy to verify that for $n \ge 1$ the sequences $(A_n)$ and $(B_n)$ satisfy the following difference equations
\begin{align} \label{eqn:dif1}
    A_n&=b_n A_{n-1}+a_nA_{n-2}\\
    B_n&=b_n B_{n-1}+a_nB_{n-2}\label{eqn:dif2}
\end{align}
with initial conditions 
\begin{equation}\label{eqn:ic}
A_{-1}=1,A_0=b_0,B_{-1}=0, \text{ and } B_0=1.
\end{equation}

When $a_n=1$ and $b_n \in \mathbb N$ for all $n\ge 1$, \eqref{eqn:cf} is called the simple continued fraction and it is a generalized continued fraction otherwise. Computing closed form solutions to difference equations as in \eqref{eqn:dif1} and \eqref{eqn:dif2} with variable coefficients $a_n,b_n$ is difficult in general \cite{mallik1998solutions}. An elementary approach is to guess the closed form of $A_n$ and $B_n$ from the first few terms and prove it using mathematical induction. Recently, this approach was used in \cite{lu2019elementary} to prove one of the many conjectures on generalized continued fraction expansion for Euler number $e$ listed in \cite{raayoni2019ramanujan}. Machine learning techniques, namely Meet-In-The-Middle and Gradient Descent, were implemented in \cite{raayoni2019ramanujan} to obtain various conjectures for mathematical constants $\pi$ and $e$ listed in \url{www.RamanujanMachine.com}.    

In this note our aim is twofold: First, we prove some of the conjectures from \url{www.RamanujanMachine.com} using similar approach to \cite{lu2019elementary} or show the equivalence of certain conjectures. Next, we describe a twisted approach to find and prove new generalized continued fractions for mathematical constants. 

For the first objective we have the following, Throughout $e=2.7182\dots$ stands for the Euler number.

\begin{theorem} \label{thm1}
We have
$$\frac{e}{2}=1+\frac{1}{2}\cplus\frac{3}{3}\cplus\frac{4}{4}\cplus \frac{5}{5}\cplus \dots.$$
\end{theorem}

\begin{theorem}\label{thm:em2}We have
$$e-2=1+\frac{-1}{1}\cplus \frac{2}{1} \cplus \frac{-1}{1} \cplus\frac{3}{1} \cplus \frac{-1}{1}\cplus \frac{4}{1} \cplus\cdots,$$
\end{theorem}

Both were conjectured in \url{www.RamanujanMachine.com}. The following simple observations can be used to prove the equivalence of two continued fractions.

\begin{theorem}\label{thm:mp} We have
$$-\left(b_0+\frac{a_1}{b_1}\cplus \frac{a_2}{b_2} \cplus \frac{a_3}{b_3} \cplus\frac{a_4}{b_4} \cplus\cdots\right)=-b_0+\frac{a_1}{-b_1}\cplus \frac{a_2}{-b_2} \cplus \frac{a_3}{-b_3}\frac{a_4}{-b_4} \cplus \cplus\cdots$$
$$b_0+\frac{a_1}{b_1}\cplus \frac{a_2}{b_2} \cplus \frac{a_3}{b_3} \cplus\frac{a_4}{b_4} \cplus\cdots=b_0+\frac{-a_1}{-b_1}\cplus \frac{-a_2}{b_2} \cplus \frac{-a_3}{-b_3} \cplus\frac{-a_4}{b_4} \cplus\cdots$$
\end{theorem}

As a corollary, we prove another conjecture from  \url{www.RamanujanMachine.com}
\begin{corollary} We have
$$-e=-3+\frac{-1}{-4}\cplus\frac{-2}{-5}\cplus\frac{-3}{-6}\cplus \dots.$$
\end{corollary}
Theorem~\ref{thm:mp} applied to the second formula in \cite{lu2019elementary} immediately gives the above corollary. Using Theorem~\ref{thm1}, Theorem~\ref{thm:mp}, and a trivial manipulation gives proof for another conjecture in \url{www.RamanujanMachine.com}.

\begin{corollary}\label{cor:1} We have
$$-\frac{e}2=-1+\frac{2}{-2}\cplus\frac{3}{-3}\cplus \frac{4}{-4} \cplus \cdots.$$
\end{corollary}
We note that using similar arguments as in the proof of Theorem~\ref{thm1}, Zhaniya Zhanabay provided another proof for Corollary~\ref{cor:1}.  Theorem~\ref{thm:em2} together with Theorem~\ref{thm:mp} immediately give

\begin{corollary} We have 
\begin{align*}
    2-e&=-1+\frac{-1}{-1}\cplus \frac{2}{-1} \cplus \frac{-1}{-1} \cplus\frac{3}{-1} \cplus \frac{-1}{-1}\cplus \frac{4}{-1} \cplus\cdots\\
    e-2&=1+\frac{1}{-1}\cplus \frac{-2}{1} \cplus \frac{1}{-1} \cplus\frac{-3}{1} \cplus \frac{1}{-1}\cplus \frac{-4}{1} \cplus\cdots
\end{align*}
\end{corollary}
Both statements were conjectured in \url{www.RamanujanMachine.com}.

\begin{corollary} \label{thm2}
The following conjectures are equivalent:
\begin{align*}
    -\frac12(\pi+2)&=-3\cdot 1+\frac{-2\cdot 1}{-3\cdot 2}\cplus \frac{-3\cdot 3 }{-3\cdot 3} \cplus \frac{-4\cdot 5}{-3\cdot 4} \cplus\frac{-5\cdot 7}{-3\cdot 5} \cplus\cdots\\
     \frac12(\pi+2)&=3\cdot 1+\frac{-2\cdot 1}{3\cdot 2}\cplus \frac{3\cdot 3 }{3\cdot 3} \cplus \frac{-4\cdot 5}{3\cdot 4} \cplus\frac{-5\cdot 7}{3\cdot 5} \cplus\cdots\\
      -\frac12(\pi+2)&=-3\cdot 1+\frac{2\cdot 1}{3\cdot 2}\cplus \frac{3\cdot 3 }{-3\cdot 3} \cplus \frac{4\cdot 5}{3\cdot 4} \cplus\frac{5\cdot 7}{-3\cdot 5} \cplus\cdots\\
      \frac12(\pi+2)&=3\cdot 1+\frac{2\cdot 1}{-3\cdot 2}\cplus \frac{-3\cdot 3 }{3\cdot 3} \cplus \frac{4\cdot 5}{-3\cdot 4} \cplus\frac{5\cdot 7}{3\cdot 5} \cplus\cdots
\end{align*}
\end{corollary}
The first two conjectures appear in \url{www.RamanujanMachine.com} and the remaining two are new. The equivalence of the four conjectures are trivial from Theorem~\ref{thm:mp}. Attempt to prove the conjectures related to continued fractions for $\pi$ using similar ideas as in the above theorems are not yet successful.  

For the second objective we obtain the following.

\begin{theorem}\label{thm:tw} Let $\sum_{k=1}^\infty c_k$ be a convergent series with nonzero terms. Define $b_1=2$, $a_1=2c_1$, and
$$b_n=\frac{2(c_n+c_{n-1})}{c_{n-1}} \text{ and } a_n=\frac{-4c_n}{c_{n-1}}, \text{ for }n \ge 2.$$
Then, we have 
$$\sum_{k=1}^\infty c_k=\frac{a_1}{b_1}\cplus \frac{a_2}{b_2} \cplus \frac{a_3}{b_3} \cplus\frac{a_4}{b_4} \cplus\cdots.$$
\end{theorem}

If we one considers the well-known series representation for $\pi/4$, namely, $\pi/4=\sum_{k=1}^\infty (-1)^{k+1}/(2k-1)$, then Theorem~\ref{thm:tw} gives 
$$\frac{\pi}4=\frac{1}{1}\cplus \frac{1^2}{2} \cplus \frac{3^2}{2} \cplus\frac{5^2}{2} \cplus\frac{7^2}{2} \cplus\cdots,$$
which is well-known (goes back to Lord Brouncker, 17th century), except the first three terms needs to be simplified.

For $(\pi-3)/4=\sum_{k=1}^\infty (-1)^{k+1}/[2k(2k+1)(2k+2)]$, another series representation of $\pi$, we obtain
$$\frac{\pi-3}4=\frac{1}{24}\cplus \frac{36}{6} \cplus \frac{5^2}{6} \cplus\frac{7^2}{6} \cplus\frac{9^2}{6} \cplus\cdots,$$
yet another well-known representation \cite{lange1999elegant}. We have the following new continued fraction for $\pi/3$:

\begin{corollary}\label{cor:cfpi} We have
$$\frac{\pi}3=\frac{-6}{-5}\cplus \frac{-75}{63} \cplus \frac{29645}{278}  \cplus \frac{a_4}{b_4}\cplus \frac{a_5}{b_5} \cplus\cdots,$$
where $a_n,b_n$ are given by
$$b_n=2(36n^2-72n+31) \text{ and } a_n=(2n-1)(2n-5)(6n-11)^2(6n-7)^2,$$
for any $n \ge 2$.
\end{corollary}

In the next section we prove Theorem~\ref{thm1}. Then, in \S~\ref{sec:tw} we introduce the twisted method to prove Theorem~\ref{thm2}.

\section{Proofs of conjectures}\label{sec:cj}

\begin{proof}[Proof of Theorem~\ref{thm1}]
We note that $b_n=n+1$ for $n \ge 0$ and $a_1=1$, $a_n=n+1$ for $n\ge 1$. Substituting into \eqref{eqn:dif1} and \eqref{eqn:dif2} we see that 
$(A_n)_{n\ge-1}=(1,   1,    3,  12,   60, 360, 2520, \dots)$ and $(B_n)_{\ge -1}=(0,   1,  2,   9, 44, 265,1854, \dots)$. From \cite{sloane2003}, \url{https://oeis.org/} we see that $A_n$ is eventually the sequence $A001710$, the number of even permutations of $n$ letters, which is given by $n!/2$. So, we derive that $A_n= (n+2)!/2$ for $n\ge 0.$ Similarly, we see that $B_n$ is the shifted sequence $A182386$ from \cite{sloane2003}. This gives $B_n=  (n+2)!\sum_{k=0}^{n+2}(-1)^k/k!$ for $n \ge 1$. It is straightforward to verify bu induction these formulas, then
$$\lim_{n \to \infty}\frac{A_n}{B_n}=\lim_{n \to \infty}\frac{1/2}{\sum_{k=0}^{n+2}(-1)^k/k!}=\frac{2^{-1}}{e^{-1}}=e/2.$$
\end{proof}

\begin{proof}[Proof of Theorem~\ref{thm:mp}] 
Note that 
\begin{align*}
    -\left(b_0+\frac{a_1}{b_1}\cplus \frac{a_2}{b_2} \cplus \frac{a_3}{b_3} \cplus\cdots\right)&=-b_0+\frac{-a_1}{b_1}\cplus \frac{a_2}{b_2} \cplus \frac{a_3}{b_3} \cplus\cdots\\
    &=-b_0+\frac{a_1}{- \left(b_1+ \frac{a_2}{b_2} \cplus \frac{a_3}{b_3} \cplus\cdots\right)}.
\end{align*}
Hence, the assertion follows by induction. For the second claim, we notice that
$$b_0+\frac{a_1}{b_1}\cplus \frac{a_2}{b_2} \cplus \frac{a_3}{b_3} \cplus\frac{a_4}{b_4} \cplus\cdots=b_0+\frac{-a_1}{-b_1+ \left(\frac{-a_2}{b_2} \cplus \frac{a_3}{b_3} \cplus\frac{a_4}{b_4} \cplus\cdots\right)}.$$
Again inductively we deduce the claim.
\end{proof}

\begin{proof}[Proof of Theorem~\ref{thm:em2}]
For $b_n=1, \ge 0$ and $(a_n)_{\ge 1}=(-1,2,-1,3,-1,4,-1,\dots)$ we get
$$(A_n)_{\ge -1}=(  1,   1,   0,   2,   2,   8,   6,  38,  32, 222, 190,\cdots),$$ $$(B_n)_{\ge -1}=(  0,   1,   1,   3,   2,  11,   9,  53,  44, 309, 265, \dots).$$
These sequences are not recognized in \cite{sloane2003}. However, if we consider subsequences with odd and even indices separately then we see for $n \ge 1$ that
\begin{align*}
    &A_{2n-1}=-n!\left(1+2\sum_{k=1}^n(-1)^k\frac1k!\right)=-n!\left(2\sum_{k=0}^n(-1)^k\frac1k!-1\right)\\
    &A_{2n}=(n+1)!+(n+2)! -2( A000166(n+1) + A000166(n+2))\\
    &B_{2n-1}=n!\sum_{k=0}^n(-1)^k\frac1k!\\
    &B_{2n} =A000166(n+1) + A000166(n+2)
\end{align*}
where $A000166(n)=n!\sum_{k=0}^n(-1)^k\frac{1}{k!}.$
We see that 
$$\lim_{n \to \infty} \frac{A_{2n-1}}{B_{2n-1}}=-2+\lim_{n \to \infty}\frac{1}{\sum_{k=0}^n(-1)^k\frac1k!}=e-2.$$ Similarly,
\begin{align*}
\lim_{n \to \infty}\frac{A_{2n}}{B_{2n}}&=\lim_{n \to \infty}\frac{(n+1)!+(n+2)!}{(n+1)!\sum_{k=0}^{n+1}(-1)^k\frac{1}{k!} + (n+2)!\sum_{k=0}^{n+2}(-1)^k\frac{1}{k!}}-2\\
&=\lim_{n \to \infty}\frac{(n+1)!+(n+2)!}{[(n+1)!+(n+2)!]\sum_{k=0}^{n+1}(-1)^k\frac{1}{k!} + (-1)^{n+2} /(n+2)!}-2\\
&=\frac{1}{\sum_{k=0}^{\infty}(-1)^k\frac{1}{k!} }-2=e-2.
\end{align*}

\end{proof}

\section{Generating new continued fractions}\label{sec:tw}

In this section we explain a general method to obtain various generalized continued fraction formulas for a given number. It is based on the following very simple idea:
\begin{enumerate}
    \item For a given number $x$, find sequences $A_n, B_n$ such that $\lim_{n \to \infty} \frac{A_n}{B_n}=x$
    \item Find $a_n,b_n$ satisfying the difference equations \eqref{eqn:dif1} and \eqref{eqn:dif2} with initial conditions \eqref{eqn:ic}.
\end{enumerate}
Then, obviously \eqref{eqn:cf} converges to $x$. To this end, we need to find formulas for $a_n,b_n$ in terms of $A_n, B_n$ and pay special attention to initial conditions as we study now. We may write \eqref{eqn:dif1} and \eqref{eqn:dif2} as a matrix equation:
\begin{equation}
\begin{bmatrix}
A_{n-1} & A_{n-2}\\
B_{n-1} & B_{n-2}
\end{bmatrix}
\begin{bmatrix}
b_n\\
a_n
\end{bmatrix}
=\begin{bmatrix}
A_n\\
B_n
\end{bmatrix}
\end{equation}
Assuming invertibility ($A_{n-1}B_{n-2}-A_{n-2}B_{n-1}\ne 0)
$ we obtain

\begin{equation*}
\begin{bmatrix}
b_n\\
a_n
\end{bmatrix}=
\frac{1}{A_{n-1}B_{n-2}-A_{n-2}B_{n-1}}\begin{bmatrix}
B_{n-2} & -A_{n-2}\\
-B_{n-1} & A_{n-1}
\end{bmatrix}\begin{bmatrix}
A_n\\
B_n
\end{bmatrix},
\end{equation*}
which summarizing gives
\begin{theorem}\label{thm:anbn} Let $x$ be a number. Let $(A_n)_{\ge -1}$ and $(B_n)_{\ge -1}$ be two sequences of numbers satisfying
\begin{enumerate}
    \item $\lim_{n \to \infty} A_n/B_n=x$, 
    \item $A_{-1}=1,B_{-1}=0, B_0=1,$
    \item $A_{n-1}B_{n-2}-A_{n-2}B_{n-1}\ne 0$ for $n \ge 2$.
\end{enumerate}
Then, for the sequences $(a_n)_{\ge 1}$ and $(b_n)_{\ge 0}$ given by
\begin{equation} \label{eqn:anbn}
    b_n:=\frac{A_nB_{n-2}-A_{n-2}B_n}{A_{n-1}B_{n-2}-A_{n-2}B_{n-1}},
    a_n:=\frac{-A_nB_{n-1}+A_{n-1}B_n}{A_{n-1}B_{n-2}-A_{n-2}B_{n-1}} \text{ for } n \ge 1,
\end{equation}
and $b_0=A_0$ we have 
$$x=b_0+\frac{a_1}{b_1}\cplus \frac{a_2}{b_2} \cplus \frac{a_3}{b_3} \cplus\frac{a_4}{b_4} \cplus\cdots.$$
\end{theorem}

We now apply Theorem~\ref{thm:anbn} to prove Theorem~\ref{thm:tw}. 

\begin{proof}[Proof of Theorem~\ref{thm:tw}]
Let $\sum_{k=1}^\infty c_k$ converge to some $x$. We let 
\begin{equation}\label{eqn:AB'}
A_n=2^n\sum_{k=1}^{n} c_k, B_n=2^n, n \ge 1 \text{ and }A_{-1}=1,A_0=0,B_{-1}=0, B_0=1.
\end{equation}
 Then, $\lim_{n \to \infty}A_n/B_n=x$ and $A_n,B_n$ satisfy the conditions of Theorem~\ref{thm:tw}. We may rewrite \eqref{eqn:anbn} as 
\begin{equation} \label{eqn:anbn''}
    b_n=\frac{\frac{A_n}{B_n}-\frac{A_{n-2}}{B_{n-2}}}{\frac{A_{n-1}}{B_{n-1}}-\frac{A_{n-2}}{B_{n-2}}} \cdot \frac{B_{n}}{B_{n-1}}\text{ and }
    a_n=\frac{-\frac{A_n}{B_n}+\frac{A_{n-1}}{B_{n-1}}}{\frac{A_{n-1}}{B_{n-1}}-\frac{A_{n-2}}{B_{n-2}}}\cdot \frac{B_{n}}{B_{n-2}}, \, n\ge 2.
\end{equation}
Using \eqref{eqn:AB'} we see that \eqref{eqn:anbn''} simplifies to
\begin{equation}\label{eqn:anbnBn}
b_n=\frac{2(c_n+c_{n-1})}{c_{n-1}} \text{ and } a_n=\frac{-4c_n}{c_{n-1}}, \text{ for }n \ge 3.
\end{equation}
The remaining cases $n=1,2$ require special attention due to initial conditions. For $n=2,$ we have 
$$b_2=\frac{2(c_1+c_2)}{c_1} \text{ and }a_2=\frac{-4c_2}{c_1}$$
For $n=1,$ we have $B_{n-2}=B_{-1}=0,$ hence we cannot use \eqref{eqn:anbn''}. Instead we use \eqref{eqn:anbn} to obtain
$b_1=2 \text{ and } a_1=2c_1.$ Applying Theorem~\ref{thm:anbn} finishes the proof.
\end{proof}

\begin{proof}[Proof of Corollary~\ref{cor:cfpi}]
Let us apply Theorem~\ref{thm:tw} to 
$$\frac{\pi}{3}=\sum_{k=1}^\infty (-1)^{k+1}\left(\frac{1}{6k-5}+\frac{1}{6k-1}\right).$$
This leads to
$$b_n=\frac{4(36n^2-72n+31)}{(6n-5)(6n-1)(2n-3)} \text{ and } a_n=\frac{(2n-1)(6n-11)(6n-7)}{(6n-5)(6n-1)(2n-3)}, n \ge 2,$$
and  $b_1=2$, $a_1=6/5$. Note that for any nonzero scalar $c$ we have
$\frac{a_1}{b_1}\cplus \frac{a_2}{b_2} \cplus \frac{a_3}{b_3}  \cplus\cdots=\frac{ca_1}{cb_1}\cplus \frac{ca_2}{b_2} \cplus \frac{a_3}{b_3}  \cplus\cdots$. Repeatedly using this fact to clear the denominators we arrive at
$$\pi/3=\frac{-6}{-5}\cplus \frac{-75}{63} \cplus \frac{29645}{278}  \cplus \frac{a_4'}{b_4'}\cplus \frac{a_5'}{b_5'} \cplus\cdots,$$
with $a_n'$ and $b_n'$ given by
$$b_n'=2(36n^2-72n+31) \text{ and } a_n'=(2n-1)(2n-5)(6n-11)^2(6n-7)^2,$$
for any $n \ge 2$.
\end{proof}

\bibliography{mybibfile}

\begin{thebibliography}{1}
\expandafter\ifx\csname url\endcsname\relax
  \def\url#1{\texttt{#1}}\fi
\expandafter\ifx\csname urlprefix\endcsname\relax\def\urlprefix{URL }\fi
\expandafter\ifx\csname href\endcsname\relax
  \def\href#1#2{#2} \def\path#1{#1}\fi

\bibitem{jones1980continued}
W.~B. Jones, W.~J. Thron, Continued fractions. analytic theory and
  applications, with a foreword by felix e. browder and an introduction by
  peter henrici, Encyclopedia of Mathematics and its Applications 11 (1980).

\bibitem{mallik1998solutions}
R.~K. Mallik, Solutions of linear difference equations with variable
  coefficients, Journal of mathematical analysis and applications 222~(1)
  (1998) 79--91.

\bibitem{lu2019elementary}
Z.~Lu, Elementary proofs of generalized continued fraction formula for $ e$,
  arXiv preprint arXiv:1907.05563 (2019).

\bibitem{raayoni2019ramanujan}
G.~Raayoni, G.~Pisha, Y.~Manor, U.~Mendlovic, D.~Haviv, Y.~Hadad, I.~Kaminer,
  The ramanujan machine: automatically generated conjectures on fundamental
  constants, arXiv preprint arXiv:1907.00205 (2019).

\bibitem{lange1999elegant}
L.~Lange, An elegant continued fraction for $\pi$, The American Mathematical
  Monthly 106~(5) (1999) 456--458.

\bibitem{sloane2003}
N.~J. Sloane, et~al., The on-line encyclopedia of integer sequences, Available
  at https://oeis.org/ (2003).

\end{thebibliography}

\end{document}